\documentclass[a4paper, 12pt]{amsproc}

\usepackage{fullpage}
\usepackage{amsmath, amsthm, amssymb, mathtools, mathrsfs}

\usepackage{hyperref}

\usepackage{chessfss}
\usepackage{mathtools}
\usepackage{tikz}
\usetikzlibrary{intersections, calc, arrows.meta}

\usepackage{pgfplots}

\usepackage{graphicx}
\usepackage{here}
\usepackage{time}

\usepackage{xcolor}
\usepackage[capitalize,nameinlink,noabbrev,nosort]{cleveref}
\hypersetup{
	colorlinks=true,       
	linkcolor=brown,          
	citecolor=brown,        
	filecolor=brown,      
	urlcolor=brown,           
}

\makeatletter
\@namedef{subjclassname@2020}{%
  \textup{2020} Mathematics Subject Classification}
\makeatother

\newtheorem{theoremcounter}{Theorem Counter}[section]

\theoremstyle{definition}

\theoremstyle{plain}
\newtheorem{lemma}[theoremcounter]{Lemma}
\newtheorem{proposition}[theoremcounter]{Proposition}
\newtheorem{corollary}[theoremcounter]{Corollary}

\newtheorem{theorem}[theoremcounter]{Theorem}

\numberwithin{equation}{section}

\newcommand{\R}{\mathbb{R}}

\DeclareMathOperator{\ReNew}{Re}
\renewcommand{\Re}{\ReNew}

\begin{document}

\title[]{On real zeros of the Hurwitz zeta function} 

\author{Karin Ikeda} 
\address{Joint Graduate School of Mathematics for Innovation, Kyushu University,
Motooka 744, Nishi-ku, Fukuoka 819-0395, Japan}
\email{ikeda.karin.236@s.kyushu-u.ac.jp}

\thanks{The first author was supported by WISE program (MEXT) at Kyushu University. 
}



\maketitle

\begin{abstract}
	In this paper, we present results on the uniqueness of the real zeros of the Hurwitz zeta function in given intervals. The uniqueness in question, if the zeros exist, has already been proved for the intervals $(0,1)$ and $(-N, -N+1)$ for $N \geq 5$ by Endo-Suzuki and Matsusaka respectively. We prove the uniqueness of the real zeros in the remaining intervals by examining the behavior of certain associated polynomials. 
 \end{abstract}

\section{Introduction}
For a fixed real number $a$ with $0<a\le 1$, the Hurwitz zeta function $\zeta(s,a)$ is defined by
$$
\zeta(s,a):=\sum_{n=0}^{\infty}\frac{1}{(n+a)^s}\quad (\Re(s)>1).
$$ 
As usual, write $s=\sigma+it\ (\sigma, t\in\R)$. The series converges absolutely and locally uniformly in the half-plane $\sigma>1$, 
and $\zeta(s,a)$ is analytically continued to the whole complex plane except for a simple pole at $s=1$. 

We are interested in the zeros of $\zeta(s,a)$. First, when $a=1/2$ or 1, the function $\zeta(s,a)$ can be expressed using 
the Riemann zeta function $\zeta(s)$ as $\zeta(s,1/2)=(2^s-1)\zeta(s),\zeta(s,1)=\zeta(s)$. Hence, apart from the non-trivial 
zeros in the critical strip $0<\sigma<1$, there is a simple zero at each even negative (non-positive when $a=1/2$) integer (called a trivial zero). 
Throughout the paper, we assume $a\ne 1/2, 1$. 

In \cite{Spira}, Spira showed that the Hurwitz zeta function $\zeta(s,a)$ has no zeros in the region 
\[ \sigma\geq1+a\] or 
\[ \sigma\leq-1\  \text{ and }\  |t|\geq1.\]
He also showed that for $|t|<1$, the only zeros of $\zeta(s,a)$ in the region 
\[ \sigma\leq \begin{cases} -4a-1 & (0<a<1/2) \\
-4a+1 & (1/2<a<1) \end{cases}\]
are on the real axis (he obtained a more precise result).

Since Spira, there has been little progress on the real zeros of the Hurwitz zeta function. Recently, however, Nakamura has given the conditions for the existence of real zeros in the interval $(0,1)$ and $(-1,0)$\,\cite{N1,N2}. This result has been improved by Matsusaka as follows.

Recall the classical result (see for example \cite{Apostol, AIK}) on the values of $\zeta(s,a)$ at non-positive integers
\[ \zeta(-N,a)=-\frac{B_{N+1}(a)}{N+1}\ \ (N\ge0). \]
Here, $B_n(x)$ is the $n$th Bernoulli polynomial defined by the generating series
\[\frac{te^{xt}}{e^t-1}=\sum_{n=0}^{\infty}B_n(x)\frac{t^n}{n!}.\]
Hence, if $B_{N}(a)$ and $B_{N+1}(a)$
have different sign, there should be at least one zero in the interval $(-N,-N+1)$. Matsusaka showed the 
converse is also true.

\begin{theorem}[\cite{MT}]\label{goma15}
Let $N\geq0$ be an integer. Then $\zeta(\sigma,a)$ has a zero in the interval $(-N,-N+1)$ if and only if $B_{N}(a)B_{N+1}(a)<0$.
\end{theorem}

Furthermore, summarizing the works of Endo-Suzuki~\cite{ES} and Matsusaka~\cite{MT}, we know the
uniqueness of zeros in the following cases.

\begin{theorem}{\cite{ES, MT}}\label{goma12}
Suppose $N\ge5$ or $N=0$. If $B_{N}(a)B_{N+1}(a)<0$, the zero of $\zeta(\sigma,a)$ in the interval $(-N,-N+1)$ 
is unique and simple.
\end{theorem}  

This paper settles completely the uniqueness of real zeros of $\zeta(s,a)$.

\begin{theorem}\label{main}
For $1\le N\le 4$, the zero of $\zeta(\sigma,a)$ in the interval $(-N,-N+1)$ 
is also unique and simple if $B_{N}(a)B_{N+1}(a)<0$.
\end{theorem}  

Thus, all real zeros of the Hurwitz zeta function, like the Riemann zeta function, are simple. Moreover, we know their locations as in the following Corollary.

\begin{corollary}\label{cor}
Let $M\ge 0$ be an integer. Then $\zeta(\sigma,a)$ has exactly one zero in the interval $[-2M-2,-2M)$.
\end{corollary}

\section{Proofs of \cref{main} and \cref{cor}}
We start with the following (more or less well-known) integral representation of the Hurwitz zeta function in a particular vertical strip (cf. \cite[Proposition~2.3]{MT} for a proof).
\begin{proposition}\label{goma21}
In the region $-N<Re(s)<-N+1\,(N\geq0)$, we have
$$
\Gamma(s)\zeta(s,a)=\int_{0}^{\infty}H_N(a,x)x^{s-1}dx,\quad H_N(a,x):=\frac{e^{(1-a)x}}{e^x-1}-\sum_{n=0}^{N}\frac{B_n(1-a)}{n!}x^{n-1}.
$$
\end{proposition}
Set
\begin{align*}
h_N(a,x)&:=x(e^x-1)H_N(a,x)\\
\intertext{and}
f_N(a,x)&:=e^{(a-1)x}\frac{\partial^{N+1}}{\partial x^{N+1}}h_N(a,x).
\end{align*}

We prove our main theorem (\cref{main}) by using the following criterion, which generalizes the previous works (\cite{ES,MT}).

\begin{proposition}\label{goma24}
If the function $\partial/\partial x f_N(a,x)$ has exactly one zero in the range $x>0$, then the Hurwitz zeta function  $\zeta(s,a)$ has a unique simple real zero in the interval $(-N,-N+1)$.
\end{proposition}

The proposition readily follows from the lemma below, which is a generalization of \cite[Lemma 2.2. 2.3]{ES}. 

\begin{lemma}\label{goma22}
Suppose $B_{N}(a)B_{N+1}(a)<0$ and the function $\partial/\partial xf_N(a,x)$ has exactly one real zero in $x>0$. Then, there exists a unique positive number $x_0=x_0(a)>0$ such that $H_N(a,x_0)=0$ and either
$$
H_N(a,x)<0\,\,(0<x<x_0)\  \text{and}\ H_N(a,x)>0\,\,(x_0<x),
$$
or
$$
H_N(a,x)>0\,\,(0<x<x_0)\  \text{and}\,\, H_N(a,x)<0\,\,(x_0<x)
$$
holds. Moreover, the function $x_0^{-\sigma}\Gamma(\sigma)\zeta(\sigma,a)$ is monotonically increasing or monotonically decreasing 
in $\sigma$ in the interval $(-N,-N+1)$.
\end{lemma}

\begin{proof}   The proof goes along the same line as in \cite{ES}. 
Since $x(e^x-1)>0$ for any $x>0$, we may look at the sign of $h_N(a,x)$. From the definition of $h_N(a,x)$, it holds
$$
h_N(a,0)=0,
$$
$$
\lim_{x\to\infty}h_N(a,x)=
\begin{cases}
\infty & (B_{N}(1-a)<0),\\
-\infty & (B_{N}(1-a)>0).
\end{cases}
$$
Also, the first derivative $\partial/\partial xh_N(a,x)$ can be calculated as
\begin{align*}
\frac{\partial}{\partial x}h_N(a,x)
&=(1+(1-a)x)e^{(1-a)x}-e^x\sum_{n=0}^{N}\frac{B_n(1-a)}{n!}x^n\\
&\quad-(e^x-1)\sum_{n=1}^{N}\frac{B_n(1-a)}{(n-1)!}x^{n-1}
\end{align*}
and from this one concludes
$$
\frac{\partial}{\partial x}h_N(a,0)=0,
$$
$$
\lim_{x\to\infty}\frac{\partial}{\partial x}h_N(a,x)=
\begin{cases}
\infty & (B_{N}(1-a)<0),\\
-\infty & (B_{N}(1-a)>0).
\end{cases}
$$
Likewise, by noting that the order of vanishing of $h_N(a,x)$ at $x=0$ is $N+2$, we similarly have
$$
\frac{\partial^i}{\partial x^i}h_N(a,0)=0,
$$
$$
\lim_{x\to\infty}\frac{\partial^i}{\partial x^i}h_N(a,x)=
\begin{cases}
\infty & (B_{N}(1-a)<0),\\
-\infty & (B_{N}(1-a)>0)
\end{cases}
$$
for $1\le i\le N+1$.
Finally, since the next order derivative
\begin{align*}
&\frac{\partial}{\partial x}f_N(a,x)\  \left(=(a-1)e^{(a-1)x}\frac{\partial^{N+1}}{\partial x^{N+1}}h_N(a,x)+e^{(a-1)x}\frac{\partial^{N+2}}{\partial x^{N+2}}h_N(a,x)\right)\\
&=(1-a)^{N+1}-\sum_{m=1}^{N}\left(\sum_{k=0}^{N-m}\binom{N+1}{k}B_{m+k}(1-a)\right)\frac{x^{m-1}}{(m-1)!}e^{ax}\\
&\quad-a\sum_{m=0}^{N}\left(\sum_{k=0}^{N-m}\binom{N+1}{k}B_{m+k}(1-a)\right)\frac{x^{m}}{m!}e^{ax},
\end{align*}
we obtain
\begin{align}
\frac{\partial}{\partial x}f_N(a,0)&=(N+2)B_{N+1}(1-a),\label{start0}\\  
\intertext{and}
\lim_{x\to\infty}\frac{\partial}{\partial x}f_N(a,x)&=
\begin{cases}
\infty & (B_{N}(1-a)<0),\\ 
-\infty & (B_{N}(1-a)>0).\label{start1}
\end{cases}
\end{align}
From this and the assumptions $B_{N}(a)B_{N+1}(a)<0$ (note $B_n(1-a)=(-1)^nB_n(a)$) and the function $\partial/\partial xf_N(a,x)$ has exactly one real zero in $x>0$, we inductively obtain the first assertion of the lemma. 

As for the second, use \cref{goma21} to obtain the expression
\begin{align*}
x_0^{-\sigma}\Gamma(\sigma)\zeta(\sigma,a)
&=x_0^{-\sigma}\int_{0}^{\infty}\left(\frac{e^{(1-a)x}}{e^x-1}-\sum_{n=0}^{N}\frac{B_n(1-a)}{n}x^{n-1}\right)x^{\sigma-1}dx\\
&=\int_{0}^{x_0}H_N(a,x)\left(\frac{x}{x_0}\right)^{\sigma}\frac{dx}{x}+\int_{x_0}^{\infty}H_N(a,x)\left(\frac{x}{x_0}\right)^{\sigma}\frac{dx}{x},
\end{align*}
from which the assertion follows easily.
\end{proof}

\begin{proof}[Proof of \cref{main}]

Let $N\geq 1$. From \cref{goma24}, we only need to show that the function $\partial/\partial x f_N(a,x)$ has exactly one zero in the range $x>0$.
We see from \eqref{start0} and \eqref{start1}  that the $\partial/\partial x f_N(a,x)$ has at least one zero in $x>0$. 

To show that it is unique, we examine the derivative $\partial^2/\partial x^2f_N(a,x)$, which is computed as (an empty sum being understood to be $0$)

\begin{align*}
&\frac{\partial^2}{\partial x^2}f_N(a,x)=\\
&e^{ax}\left(-\sum_{m=0}^{N-2}\Biggl(\sum_{k=0}^{N-2-m}\binom{N+1}{k}B_{m+2+k}(1-a)+2a\sum_{k=0}^{N-1-m}\binom{N+1}{k}B_{m+1+k}(1-a) \right.\\
&+a^2\sum_{k=0}^{N-m}\binom{N+1}{k}B_{m+k}(1-a)\Biggr)\frac{x^m}{m!}\\
&+\left(-2aB_{N}(1-a)-a^2\left((N+1)B_{N}(1-a)+B_{N-1}(1-a)\right)\right)\frac{x^{N-1}}{(N-1)!}\\
&\left. -a^2B_{N}(1-a)\frac{x^{N}}{N!}\right),
\end{align*}
so that we look at the polynomial $e^{-ax}\partial^2/\partial x^2f_N(a,x)$ in $x$ of degree $N$. 
Set $P_N(a,x):=e^{-ax}\partial^2/\partial x^2f_N(a,x)$ and let  $C_{N,m}(a)$
be the coefficient of $x^m$ in $P_N(a,x)$:
\[ P_N(a,x)=\sum_{m=0}^{N} C_{N,m}(a)\,x^m. \]
 Explicitly, the coefficients are given as
\begin{align*}
C_{N,m}(a)&=
 -\Biggl(\sum_{k=0}^{N-2-m}\binom{N+1}{k}B_{m+2+k}(1-a)+2a\sum_{k=0}^{N-1-m}\binom{N+1}{k}B_{m+1+k}(1-a)\\
 &+a^2\sum_{k=0}^{N-m}\binom{N+1}{k}B_{m+k}(1-a)\Biggr)\frac{1}{m!}\quad(0\le m\le N).
\end{align*}

We will show that the polynomial $P_N(a,x)$ has either no or one real zero in the range $x>0$.
Then we can conclude, from the sign of $\partial/\partial xf_N(a,0)$ and the behavior of $\partial/\partial xf_N(a,x)$ as $x\to \infty$,
that the $\partial/\partial xf_N(a,x)$ has only one zero in $x>0$.

Now we look at each case of $N$ separately.\\

\textbf{The case of $N=1$.}
This case is trivial because $P_1(a,x)$ is a linear polynomial.\\

\textbf{The case of $N=2$.}
In this case, the $P_2(a,x)$ is a polynomial of degree $2$ in $x$, and their coefficients are given as
\begin{align*}
C_{2,0}(a)&=-a^2+(2a+3a^2)B_1(a)-(1+6a+3a^2)B_2(a),\\
C_{2,1}(a)&=a^2B_1(a)+\left(-2a-3a^2\right)B_2(a),\\
C_{2,2}(a)&=-\frac{1}{2}a^2B_2(a).
\end{align*}
Let $\alpha$ and $\beta$ be two roots of $P_2(a,x)$, then we see that
\begin{align}\label{goma1}
\alpha+\beta>0\Longleftrightarrow C_{2,1}(a)C_{2,2}(a)<0,\qquad\alpha\beta>0\Longleftrightarrow C_{2,0}(a)C_{2,2}(a)>0.
\end{align}

Now, we examine the sign changes of each polynomials (in $a$) $C_{2,0}(a),\,C_{2,1}(a),\,C_{2,2}(a)$.  First, consider $C_{2,0}(a)$. From the numerical values
\begin{align*}
&\mu_1=C_{2,0}(-1.251)=-0.00334706,\qquad \nu_1=C_{2,0}(-1.250)=0.00911458,\\
&\mu_2=C_{2,0}(-0.115)=0.000708631,\qquad \,\nu_2=C_{2,0}(-0.114)=-0.00118935,\\
&\mu_3=C_{2,0}(0.402)=-0.000598225,\qquad\, \nu_3=C_{2,0}(0.403)=0.000839283,\\
&\mu_4=C_{2,0}(0.962)=0.00376954,\qquad\quad\,\,\, \nu_4=C_{2,0}(0.963)=-0.000230453,
\end{align*}
we see that $C_{2,0}(a)$ has at least one root in each interval $(\mu_j,\nu_j)$ for $1\le j\le4$.  But since $C_{2,0}(a)$ is of
degree $4$, these are exactly the four roots. We also compute the approximate zeros of  $C_{2,0}'(a)$:
\begin{align*}
&C_{2,0}'(-0.873)=0.000063404,\qquad C_{2,0}'(-0.872)=-0.0193418,\\
&C_{2,0}'(0.128)=-0.00116582,\qquad\,\, C_{2,0}'(0.129)=0.00623973,\\
&C_{2,0}'(0.744)=0.0100306,\qquad \quad\,\, \,\,C_{2,0}'(0.745)=-0.0019235.
\end{align*}
From these, we obtain the following table.

\begin{center}
\begin{tabular}{|c||ccccccccccc|}
\hline
$a$ & $0$ & $\cdots$ & $0.128\ldots$ & $\cdots$ & $0.402\ldots$ & $\cdots$ & $0.744\ldots$ & $\cdots$ & $0.962\ldots$ & $\cdots$ & $1$\\
\hline
$C_{2,0}^{'}(a)$ & $-$ & $-$ & $0$ & $+$ & $+$ & $+$ & $0$ & $-$ & $-$ & $-$ & $-$\\
\hline
$C_{2,0}(a)$ & $-1/6$ & $\searrow$ & & $\nearrow$ & $0$ & $\nearrow$ & & $\searrow$ & $0$ & $\searrow$ & \\
\hline
\end{tabular}
\end{center}
We obtain the similar tables for $C_{2,1}(a),C_{2,2}(a)$ as follows.

\begin{center}
\begin{tabular}{|c||ccccccccccc|}
\hline
$a$ & $0$ & $\cdots$ & $0.129\ldots$ & $\cdots$ & $0.253\ldots$ & $\cdots$ & $0.684\ldots$ & $\cdots$ & $0.899\ldots$ & $\cdots$ & $1$\\
\hline
$C_{2,1}^{'}(a)$ & $-$ & $-$ & $0$ & $+$ & $+$ & $+$ & $0$ & $-$ & $-$ & $-$ & $-$\\
\hline
$C_{2,1}(a)$ & $0$ & $\searrow$ & & $\nearrow$ & $0$ & $\nearrow$ & & $\searrow$ & $0$ & $\searrow$ & \\
\hline
\end{tabular}
\end{center}

\begin{center}
\begin{tabular}{|c||ccccccccccc|}
\hline
$a$ & $0$ & $\cdots$ & $0.135\ldots$ & $\cdots$ & $0.211\ldots$ & $\cdots$ & $0.614\ldots$ & $\cdots$ & $0.788\ldots$ & $\cdots$ & $1$\\
\hline
$C_{2,2}^{'}(a)$ & $-$ & $-$ & $0$ & $+$ & $+$ & $+$ & $0$ & $-$ & $-$ & $-$ & $-$\\
\hline
$C_{2,2}(a)$ & $0$ & $\searrow$ & & $\nearrow$ & $0$ & $\nearrow$ & & $\searrow$ & $0$ & $\searrow$ & \\
\hline
\end{tabular}
\end{center}
Based on these, the rough graphs of  $C_{2,0}(a),C_{2,1}(a),C_{2,2}(a)$ can be drawn as follows.\\

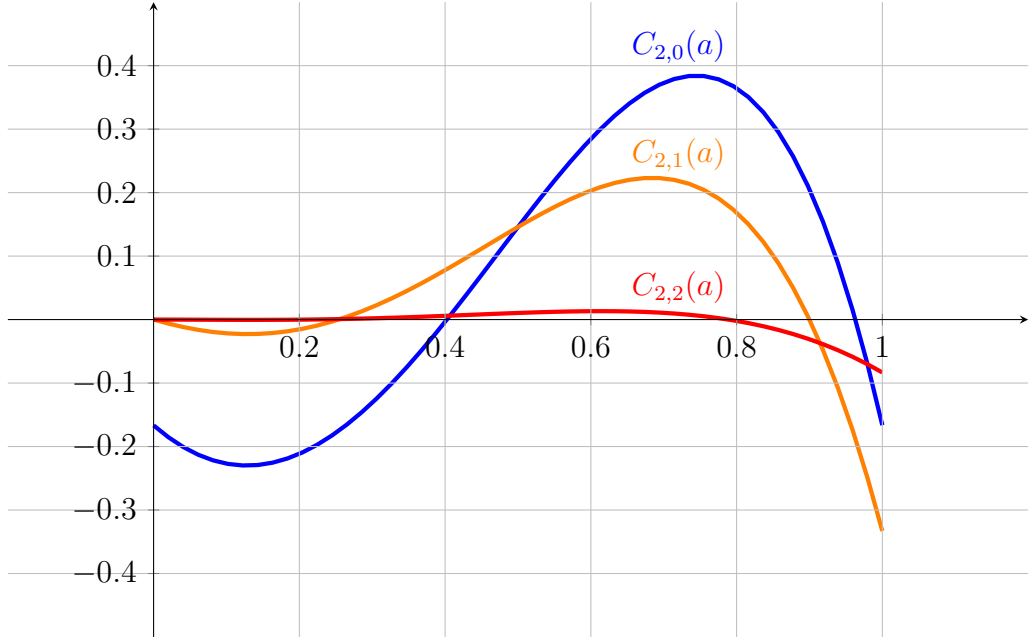
\begin{tikzpicture}
\begin{axis}[
        axis y line=center,
        axis x line=middle, 
        axis on top=true,
        xmin=-0.2,
        xmax=1.2,
        ymin=-0.5,
        ymax=0.5,
        height=10.0cm,
        width=15.0cm,
        grid,
        xtick={0,0.2,...,1},
        ytick={-0.4,-0.3,...,0.4},
    ]
    \addplot [domain=0:1, samples=50, mark=none, ultra thick, blue] {-3*(x)^4 + 4*(x)^2 - (x) - 1/6};
    \addplot [domain=0:1, samples=50, mark=none, ultra thick, orange] {-3*(x)^4 + 2*(x)^3 + (x)^2 - 1/3*(x)};
    \addplot [domain=0:1, samples=50, mark=none, ultra thick, red] {-1/2*(x)^4 + 1/2*(x)^3 - 1/12*(x)^2};
    \node [left, blue] at (axis cs: 0.8,0.43) {$C_{2,0}(a)$};
    \node [left, orange] at (axis cs: 0.8,0.26) {$C_{2,1}(a)$};
    \node [left, red] at (axis cs: 0.8,0.05) {$C_{2,2}(a)$};
\end{axis}
\end{tikzpicture}

Let $c_{2,0,i}, c_{2,1,i}, c_{2,2,i}\ (i=1,2)$ be the zeros of $C_{2,0}(a), C_{2,1}(a)$, and $C_{2,2}(a)$ in $0<a<1$.
We see that 
\[ 0<c_{2,2,1}<c_{2,1,1}<c_{2,0,1} <c_{2,2,2}<c_{2,1,2}<c_{2,0,2}<1 \]
and, if $a$ is in the range
\[ 0<a<c_{2,2,1}\,\text{ or } c_{2,0,1}<a<c_{2,2,2}\,\text{ or } c_{2,0,2}<a<1, \]
then all coefficients of $P_2(a,x)$ have the same signs and so $P_2(a,x)=0$ has no solution in $x>0$. If $a$ satisfies 
\[ c_{2,2,1}<a<c_{2,0,1}\,\text{ or } c_{2,2,2}<a<c_{2,0,2}, \]
then we have $\alpha \beta<0$ and $P_2(a,x)=0$ has one negative solution and one positive solution.

We therefore have shown that in all cases $P_2(a,x)$ has at most one real zero in the range $x>0$.\\

\textbf{The case of $N=3$.}
We proceed in the same way as in the case of $N=2$. The $P_3(a,x)$ is a polynomial of degree $3$ in $x$. Let $\alpha$, $\beta$ and $\gamma$ be three roots of $P_3(a,x)$, then we see that

\begin{align}\label{goma2}
\begin{split}
&\alpha+\beta+\gamma>0\Longleftrightarrow C_{3,2}(a)C_{3,3}(a)<0,\\
&\alpha\beta+\beta\gamma+\gamma\alpha>0\Longleftrightarrow C_{3,1}(a)C_{3,3}(a)>0,\\
&\alpha\beta\gamma>0\Longleftrightarrow C_{3,0}(a)C_{3,3}(a)<0.
\end{split}
\end{align}

We obtain the following table and rough graphs of $C_{3,0}(a)$, $C_{3,1}(a)$, $C_{3,2}(a)$, $C_{3,3}(a)$.

\begin{flushleft}
\begin{tabular}{|c||ccccccccccc|}
\hline
$a$ & $0$ & $\cdots$ & $0.141\ldots$ & $\cdots$ & $0.440\ldots$ & $\cdots$ & $0.696\ldots$ & $\cdots$ & $0.976\ldots$ & $\cdots$ & $1$\\
\hline
$C_{3,0}^{'}(a)$ & $+$ & $+$ & $+$ & $+$ & $0$ & $-$ & $-$ & $-$ & $0$ & $+$ & $+$\\
\hline
$C_{3,0}(a)$ & $-2$ & $\nearrow$ & $0$ & $\nearrow$ & & $\searrow$ & $0$ & $\searrow$ & & $\nearrow$ & \\
\hline
\end{tabular}
\end{flushleft}

\begin{flushleft}
\begin{tabular}{|c||ccccccccc|}
\hline
$a$ & $0$ & $\cdots$ & $0.364\ldots$ & $\cdots$ & $0.622\ldots$ & $\cdots$ & $0.946\ldots$ & $\cdots$ & $1$\\
\hline
$C_{3,1}^{'}(a)$ & $+$ & $+$ & $0$ & $-$ & $-$ & $-$ & $0$ & $+$ & $+$\\
\hline
$C_{3,1}(a)$ & $0$ & $\nearrow$ & & $\searrow$ & $0$ & $\searrow$ & & $\nearrow$ & \\
\hline
\end{tabular}
\end{flushleft}

\begin{flushleft}
\begin{tabular}{|c||ccccccccc|}
\hline
$a$ & $0$ & $\cdots$ & $0.361\ldots$ & $\cdots$ & $0.542\ldots$ & $\cdots$ & $0.896\ldots$ & $\cdots$ & $1$\\
\hline
$C_{3,2}^{'}(a)$ & $+$ & $+$ & $0$ & $-$ & $-$ & $-$ & $0$ & $+$ & $+$\\
\hline
$C_{3,2}(a)$ & $0$ & $\nearrow$ & & $\searrow$ & $0$ & $\searrow$ & & $\nearrow$ & \\
\hline
\end{tabular}
\end{flushleft}

\begin{flushleft}
\begin{tabular}{|c||ccccccccc|}
\hline
$a$ & $0$ & $\cdots$ & $0.355\ldots$ & $\cdots$ & $0.5$ & $\cdots$ & $0.844\ldots$ & $\cdots$ & $1$\\
\hline
$C_{3,3}^{'}(a)$ & $+$ & $+$ & $0$ & $-$ & $-$ & $-$ & $0$ & $+$ & $+$\\
\hline
$C_{3,3}(a)$ & $0$ & $\nearrow$ & & $\searrow$ & $0$ & $\searrow$ & & $\nearrow$ & \\
\hline
\end{tabular}
\end{flushleft}


\begin{tikzpicture}
\begin{axis}[
        axis y line=center,
        axis x line=middle, 
        axis on top=true,
        xmin=-0.2,
        xmax=1.2,
        ymin=-0.5,
        ymax=0.5,
        height=10.0cm,
        width=15.0cm,
        grid,
        xtick={0,0.2,...,1},
        ytick={-0.4,-0.3,...,0.4},
    ]
    \addplot [domain=0:1, samples=50, mark=none, ultra thick, blue] {48*(x)^5 - 120*(x)^3 + 60*(x)^2 + 8*(x) - 2};
    \addplot [domain=0:1, samples=50, mark=none, ultra thick, orange] {72*(x)^5 - 60*(x)^4 - 60*(x)^3 + 40*(x)^2 + 2*(x)};
    \addplot [domain=0:1, samples=50, mark=none, ultra thick, red] {24*(x)^5 - 30*(x)^4 + 5*(x)^2};
    \addplot [domain=0:1, samples=50, mark=none, ultra thick, teal] {2*(x)^5 - 3*(x)^4 + (x)^3};
    \node [left, blue] at (axis cs: 0.84,0.43) {$C_{3,0}(a)$};
    \node [left, orange] at (axis cs: 0.6,0.35) {$C_{3,1}(a)$};
    \node [left, red] at (axis cs: 0.4,0.35) {$C_{3,2}(a)$};
    \node [left, teal] at (axis cs: 0.4,0.05) {$C_{3,3}(a)$};
\end{axis}
\end{tikzpicture}

Let  $c_{3,0,1}, c_{3,0,2},  c_{3,1,1}, c_{3,2,1}, c_{3,3,1}$ be the zeros of $C_{3,0}(a)$, $C_{3,1}(a)$, $C_{3,2}(a)$ and $C_{3,3}(a)$ in $0<a<1$. We see that 
$$
0<c_{3,0,1}<c_{3,3,1}<c_{3,2,1}<c_{3,1,1}<c_{3,0,2}<1.
$$
If $a$ is in the range $c_{3,0,1}<a<c_{3,3,1}\,\text{ or }c_{3,0,2}<a<1$ then $P_3(a,x)=0$ has no solution in $x>0$, because $C_{3,0}(a)$, $C_{3,1}(a)$, $C_{3,2}(a)$, and $C_{3,3}(a)$ are all positive or all negative.

In the case of $0<a<c_{3,0,1}$ or $c_{3,1,1}<a<c_{3,0,2}$, the sign of the constant term $C_{3,0}(a)$ is different from these of the remaining coefficients $C_{3,1}(a)$, $C_{3,2}(a)$, and $C_{3,3}(a)$. We thus conclude that there is only one solution of $P_3(a,x)=0$ in the range $x>0$ in these cases. 

Finally, suppose $a$ is in the range 
$$
c_{3,3,1}<a<c_{3,1,1}.
$$
Then we have $\alpha \beta \gamma >0$. If all $\alpha, \beta,$ and $\gamma$ are real, then, together with 
$\alpha\beta+\beta\gamma+\gamma\alpha < 0$, we see that only one of $\alpha, \beta, \gamma$ is positive,
i.e.,  $P_3(a,x)=0$ has two negative solutions and one positive solution. If two of $\alpha, \beta, \gamma$ are
complex (conjugate with each other), we see from $\alpha \beta \gamma >0$ that the other one is the only positive root.

Hence we  have shown that $P_3(a,x)$ has at most one real zero in the range $x>0$.\\

\textbf{The case of $N=4$.}
The argument is also similar. The $P_4(a,x)$ is a polynomial of degree $4$ in $x$. Notice that $C_{4,1}(a)$, $C_{4,2}(a)$, and $C_{4,3}(a)$  have two imaginary zeros, but by considering the zeros of the derivative $C_{4,1}(a)$, $C_{4,2}(a)$, and  $C_{4,3}(a)$ and the signs of their values at $a=0$, then we obtain the following table and graphs.
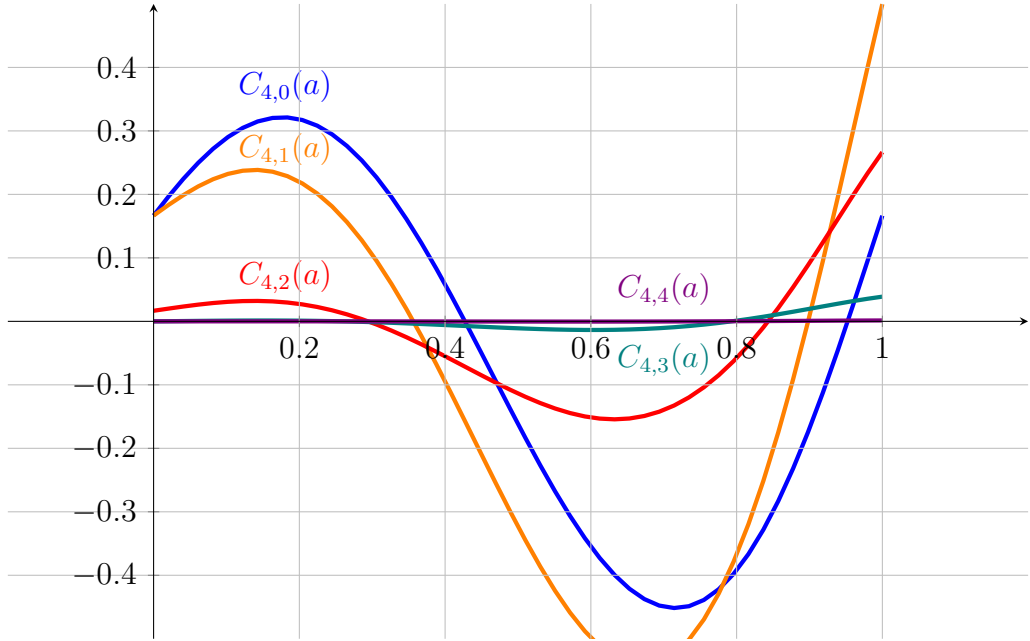
\begin{flushleft}
\begin{tabular}{|c||ccccccccccc|}
\hline
$a$ & $0$ & $\cdots$ & $0.176\ldots$ & $\cdots$ & $0.427\ldots$ & $\cdots$ & $0.715\ldots$ & $\cdots$ & $0.952\ldots$ & $\cdots$ & $1$\\
\hline
$C_{4,0}^{'}(a)$ & $+$ & $+$ & $0$ & $-$ & $-$ & $-$ & $0$ & $+$ & $+$ & $+$ & $+$\\
\hline
$C_{4,0}(a)$ & $1/6$ & $\nearrow$ & & $\searrow$ & $0$ & $\searrow$ & & $\nearrow$ & $0$ & $\nearrow$ & \\
\hline
\end{tabular}
\end{flushleft}

\begin{flushleft}
\begin{tabular}{|c||ccccccccccc|}
\hline
$a$ & $0$ & $\cdots$ & $0.138\ldots$ & $\cdots$ & $0.356\ldots$ & $\cdots$ & $0.757\ldots$ & $\cdots$ & $0.898\ldots$ & $\cdots$ & $1$\\
\hline
$C_{4,1}^{'}(a)$ & $+$ & $+$ & $0$ & $-$ & $-$ & $-$ & $0$ & $+$ & $+$ & $+$ & $+$\\
\hline
$C_{4,1}(a)$ & $1/6$ & $\nearrow$ & & $\searrow$ & $0$ & $\searrow$ & & $\nearrow$ & $0$ & $\nearrow$ & \\
\hline
\end{tabular}
\end{flushleft}

\begin{flushleft}
\begin{tabular}{|c||ccccccccccc|}
\hline
$a$ & $0$ & $\cdots$ & $0.138\ldots$ & $\cdots$ & $0.294\ldots$ & $\cdots$ & $0.632\ldots$ & $\cdots$ & $0.843\ldots$ & $\cdots$ & $1$\\
\hline
$C_{4,2}^{'}(a)$ & $+$ & $+$ & $0$ & $-$ & $-$ & $-$ & $0$ & $+$ & $+$ & $+$ & $+$\\
\hline
$C_{4,2}(a)$ & $1/60$ & $\nearrow$ & & $\searrow$ & $0$ & $\searrow$ & & $\nearrow$ & $0$ & $\nearrow$ & \\
\hline
\end{tabular}
\end{flushleft}

\begin{flushleft}
\begin{tabular}{|c||ccccccccccc|}
\hline
$a$ & $0$ & $\cdots$ & $0.151\ldots$ & $\cdots$ & $0.259\ldots$ & $\cdots$ & $0.603\ldots$ & $\cdots$ & $0.793\ldots$ & $\cdots$ & $1$\\
\hline
$C_{4,3}^{'}(a)$ & $+$ & $+$ & $0$ & $-$ & $-$ & $-$ & $0$ & $+$ & $+$ & $+$ & $+$\\
\hline
$C_{4,3}(a)$ & $0$ & $\nearrow$ & & $\searrow$ & $0$ & $\searrow$ & & $\nearrow$ & $0$ & $\nearrow$ & \\
\hline
\end{tabular}
\end{flushleft}

\begin{flushleft}
\begin{tabular}{|c||ccccccccccc|}
\hline
$a$ & $0$ & $\cdots$ & $0.162\ldots$ & $\cdots$ & $0.240\ldots$ & $\cdots$ & $0.588\ldots$ & $\cdots$ & $0.759\ldots$ & $\cdots$ & $1$\\
\hline
$C_{4,4}^{'}(a)$ & $+$ & $+$ & $0$ & $-$ & $-$ & $-$ & $0$ & $+$ & $+$ & $+$ & $+$\\
\hline
$C_{4,4}(a)$ & $0$ & $\nearrow$ & & $\searrow$ & $0$ & $\searrow$ & & $\nearrow$ & $0$ & $\nearrow$ & \\
\hline
\end{tabular}
\end{flushleft}

\begin{tikzpicture}
\begin{axis}[
        axis y line=center,
        axis x line=middle, 
        axis on top=true,
        xmin=-0.2,
        xmax=1.2,
        ymin=-0.5,
        ymax=0.5,
        height=10.0cm,
        width=15.0cm,
        grid,
        xtick={0,0.2,...,1},
        ytick={-0.4,-0.3,...,0.4},
    ]
    \addplot [domain=0:1, samples=50, mark=none, ultra thick, blue] {-5*(x)^6 + 20*(x)^4 - 15*(x)^3 - 3/2*(x)^2 + 3/2*(x) + 1/6};
    \addplot [domain=0:1, samples=50, mark=none, ultra thick, orange] {-10*(x)^6 + 10*(x)^5 + 15*(x)^4 - 15*(x)^3 - 1/2*(x)^2 + 5/6*(x) + 1/6};
    \addplot [domain=0:1, samples=50, mark=none, ultra thick, red] {-5*(x)^6 + 15/2*(x)^5 + 5/4*(x)^4 - 15/4*(x)^3 + 1/12*(x)^2 + 1/6*(x) + 1/60};
    \addplot [domain=0:1, samples=50, mark=none, ultra thick, teal] {-5/6*(x)^6 + 3/2*(x)^5 - 5/12*(x)^4 - 1/4*(x)^3 + 1/36*(x)^2 +1/90*(x)};
    \addplot [domain=0:1, samples=50, mark=none, ultra thick, violet] {-1/24*(x)^6 + 1/12*(x)^5 - 1/24*(x)^4 + 1/720*(x)^2};
    \node [left, blue] at (axis cs: 0.26,0.37) {$C_{4,0}(a)$};
    \node [left, orange] at (axis cs: 0.26,0.27) {$C_{4,1}(a)$};
    \node [left, red] at (axis cs: 0.26,0.07) {$C_{4,2}(a)$};
    \node [left, teal] at (axis cs: 0.78,-0.06) {$C_{4,3}(a)$};
    \node [left, violet] at (axis cs: 0.78,0.05) {$C_{4,4}(a)$};
\end{axis}
\end{tikzpicture}

Let $c_{4,0,i}, c_{4,1,i}, c_{4,2,i}, c_{4,3,i}, c_{4,4,i}\ (i=1,2)$ be the zeros of $C_{4,0}(a), C_{4,1}(a), C_{4,2}(a), C_{4,3}(a)$, and $C_{4,4}(a)$ in $0<a<1$.
We see that 
\[ 0<c_{4,4,1}<c_{4,3,1}<c_{4,2,1} <c_{4,1,1}<c_{4,0,1}<c_{4,4,2}<c_{4,3,2}<c_{4,2,2}<c_{4,1,2}<c_{4,0,2}<1. \]

We can easily see the following as before. If $a$ is in the range
$$
0<a<c_{4,4,1},\,c_{4,0,1}<a<c_{4,4,2},\,c_{4,0,2}<a<1
$$
then $P_4(a,x)=0$ has no real zero in the range $x>0$, and if $a$ is in
$$
c_{4,1,1}<a<c_{4,0,1},\,c_{4,1,2}<a<c_{4,0,2}
$$
then $P_4(a,x)=0$ has  only one solution in the range $x>0$.

Now, we only need to show that $P_4(a,x)$ has at most one real zero in the range $x>0$ in the case $c_{4,4,1}<a<c_{4,1,1},\,c_{4,4,2}<a<c_{4,1,2}$.
 In this case, $C_{4,0}(a)$ and $C_{4,4}(a)$ have different signs, so the sign of the product of all four roots of $P_4(a,x)$ is negative.
 If there are complex roots, then we see from this that there is only one positive root. Now suppose all roots are
 real.  Note the signs of $P_4(a,x)$ at $x=0$ and at $x\rightarrow\infty$ are different. Therefore, it is sufficient to show that $\partial/\partial x P_4(a,x)=0$ has at most one real zero in the range $x>0$. Since
$$
\frac{\partial}{\partial x}P_4(a,x)=4C_{4,4}(a)x^3+3C_{4,3}(a)x^2+2C_{4,2}(a)x+C_{4,1}(a),
$$
the same argument gives us the following. If $a$ is in the range
$$
c_{4,2,1}<a<c_{4,1,1}\  \text{or}\,\, c_{4,2,2}<a<c_{4,1,2},
$$
then the constant term $C_{4,1}(a)$ has different sign from other $C_{4,i}\ (i=2,3,4)$ and hence $\partial/\partial xP_4(a,x)=0$ has only one solution in the range $x>0$. If $a$ is in the range
$$
c_{4,4,1}<a<c_{4,2,1}\  \text{or}\,\, c_{4,4,2}<a<c_{4,2,2},
$$
then it holds $C_{4,4}(a)C_{4,1}(a)<0$ and $C_{4,4}(a)C_{4,2}<0$, thus $\partial/\partial xP_4(a,x)=0$ has only one solution in the range $x>0$ by the same argument as before. We therefore have shown that $P_4(a,x)$ has at most one real zero in the range $x>0$.
\end{proof}

\begin{proof}[Proof of \cref{cor}]
When $a\ne 1/2, 1$, $B_{2M+1}(a)B_{2M+3}(a)<0$ holds for any integer $M\ge0$. This together with \cref{goma12} and \cref{main} immediately gives the corollary.
\end{proof}


\section{Alternative proof of the uniqueness in the interval $(-4,-3)$}
For the interval $(-4,-3)$, we may proceed as follows using a result of Spira \cite[Theorem~3]{Spira}.

Suppose $B_4(a)B_5(a)<0$. Then $\zeta(s,a)$ has odd number of real zeros in the interval $(-4,-3)$.
 If we denote by $b_{4,1}$ and $b_{4,2}$ the two roots of $B_4(a)$ in $(0,1)$ ($b_{4,1}=0.2403\ldots<b_{4,2}=0.796\ldots$),
the condition  $B_4(a)B_5(a)<0$ is equivalent to $b_{4,1}<a<1/2$ or $b_{4,2}<a<1$.
 First, consider the case $1/4\le a<1/2$ or $b_{4,2}<a<1$. Spira's theorem shows that $\zeta(s,a)$ has exactly one real zero in the interval of width $2$ that contains the interval $(-4,-3)$. Therefore, the zero in $(-4,-3)$ is unique. 
 
In the case of $b_{4,1}<a<1/4$, Spira's theorem tells us that there are at most two real zeros of  $\zeta(s,a)$ 
in the interval $(-4,-3)$ (two length 2 intervals covers $(-4,-3)$). But since we know that the number of zeros is odd, we conclude there is only one.

\section*{Acknowledgements}
The author would like to express her sincere gratitude to Professors Masanobu Kaneko, Takashi Nakamura, and Toshiki Matsusaka for their helpful advice and comments. This work is supported by WISE program (MEXT) at Kyushu University.



\end{document}